\documentclass[11pt]{article}
\usepackage{graphicx}

\title{Forcing $k$-repetitions in degree sequences}

\author{
Y. Caro
\thanks{Department of exact sciences, University of Haifa at Oranim, Tivon 36006, Israel.
email: yacaro@kvgeva.org.il}
\and
A. Shapira
\thanks{School of Mathematics, Tel-Aviv University, Tel-Aviv, Israel 69978. Email: {\tt asafico@tau.ac.il}. Supported in part by ISF Grant 224/11 and a Marie-Curie CIG Grant 303320.}
\and
R. Yuster
\thanks{Department of Mathematics, University Haifa, Haifa 31905, Israel.
email: raphy@math.haifa.ac.il}
}

\date{}

\newtheorem{theorem}{Theorem}[section]
\newtheorem{prop}[theorem]{Proposition}
\newtheorem{lemma}[theorem]{Lemma}
\newtheorem{corollary}[theorem]{Corollary}

\newenvironment{proof}{\noindent{\textbf{Proof:}}}{$\qed$\vskip\belowdisplayskip}
\newenvironment{prevproof}[2]{\noindent {\bf {Proof ({#1}~\ref{#2}):}}}{$\qed$\vskip \belowdisplayskip}

\newcommand{\qed}{\hspace*{\fill} \rule{7pt}{7pt}\vspace{10pt}}

\begin{document}
\maketitle
\setcounter{page}{1}
\begin{abstract}

One of the most basic results in graph theory states that every graph with at least two vertices has two vertices with the same degree.
Since there are graphs without $3$ vertices of the same degree, it is natural to ask if for any fixed $k$, every graph $G$ is ``close'' to a graph $G'$ with 
$k$ vertices of the same degree. Our main result in this paper is that this is indeed the case. Specifically, we show that
for any positive integer $k$, there is a constant $C=C(k)$, so that given any graph $G$, one can remove from $G$ at most $C$
vertices and thus obtain a new graph $G'$ that contains at least $\min\{k,|G|-C\}$ vertices of the same degree. 

Our main tool is a multidimensional zero-sum theorem for integer sequences, which we prove using an old geometric approach of Alon and Berman.

\end{abstract}

\section{Introduction}

All graphs considered in this paper are finite, simple, and undirected.
Graph theory notation follows \cite{bollobas-1978}.

Perhaps one of the most obvious properties shared by all graphs (with at least two vertices) is that they have at least two vertices
with the same degree. In general, the {\em repetition number} of a graph $G$, denoted by $rep(G)$, is
the maximum multiplicity of a vertex degree in $G$.
So, any graph with at least two vertices has $rep(G) \ge 2$, and there are simple constructions showing that
equality holds for infinitely many graphs.
Repetition numbers of graphs (and hypergraphs) have been studied by several researchers.
See \cite{BBLM-2013,bollobas-1996,BS-1997,CW-2009,CY-2010,CHPSW-2001,CS-1993} for some representing works.

Since infinitely many graphs have $rep(G)=2$, it seems interesting to determine how many vertices one must delete from a graph in order
to increase its repetition number to $3$ or higher. We state the problem more formally as follows.
For a given positive integer $k$, let $C=C(k)$ denote the least integer such that any graph with $n$ vertices
has an induced subgraph with at least $n-C$ vertices whose repetition number is at least $\min\{k,n-C\}$.
Stated otherwise, we can delete at most $C(k)$ vertices from any graph and obtain a subgraph with $k$ repeated degrees
(or a regular subgraph in case the subgraph has less than $k$ vertices).
Trivially $C(1)=C(2)=0$. A simple lower bound of $C(k) \ge k-2$ is easily obtained from the fact that there are infinitely
many graphs with $rep(G)=2$ and precisely one pair of vertices with the same degree.
For general $k$, it is not obvious, however, that $C(k)$ is finite.
The main result of this paper shows that it is.

Before we state our main result, let us consider a more general setting, where the graph in question is the complete graph
with integer edge weights in $\{0,\ldots,r\}$. The (weighted) degree of a vertex $v$, denoted by $deg(v)$, is the sum
of the weights of the edges incident with it. Likewise, the repetition number of an edge-weighted graph is the
maximum multiplicity of a vertex degree.
Observe that, unlike the unweighted case, already for $r=2$ we do not necessarily have two vertices with the
same degree, as can be seen by assigning the weights $0,1,2$ to the edges of a triangle.
The related well-studied graph parameter of {\em irregularity strength} asks for the smallest $r$ for which the edges of
a graph can be weighted with $\{1,\ldots,r\}$ (weight $0$ represents non-edges) such that all vertex degrees are distinct
(see \cite{KKP-2011,seamone-2012}).

For given positive integers $k,r$, let $C=C(k,r)$ denote the least integer such that any complete graph with $n$ vertices
and edge weights in $\{0,\ldots,r\}$ has a complete subgraph with at least $n-C$ vertices whose repetition number is at least $\min\{k,n-C\}$.
Notice that $C(k)=C(k,1)$. Our main result is then the following.

\begin{theorem}
\label{t:1}
For positive integers $k,r$ there exists a constant $C=C(k,r)$ such that the following holds.
Any complete graph with $n$ vertices whose edges have weights in $\{0,\ldots,r\}$ contains a complete subgraph
with at least $n-C$ vertices whose repetition number is at least $\min\{k,n-C\}$.
\end{theorem}

The bound we obtain for $C(k,r)$ via the proof of Theorem \ref{t:1} is exponential in terms of $k$ already for $r=1$.
In the case $r=1$, it satisfies $C(k)=C(k,1) \le (8k)^k$. Together with a nontrivial lower bound of $C(k) \ge \Omega(k \log k)$
that we prove in Section \ref{s:lower} we obtain the following corollary.

\begin{corollary}\label{c:bound} We have $$\Omega(k \log k) \le C(k) \le (8k)^k\;.$$
\end{corollary}

It would be of course interesting to close the exponential gap between the upper bound and lower bound in the above corollary.
Already for the first nontrivial case, $C(3)$, the proof of Theorem \ref{t:1} gives the explicit bound $C(3) \le 203$.
An alternative ad hoc argument shows that $3 \le C(3) \le 6$ (see Section \ref{s:c3}).
Even the exact value of $C(3)$ is an open problem. When trying to adapt these ad hoc arguments to obtain an upper bound for $C(4)$
one quickly sees that the case analysis becomes significantly involved.
Instead, our proof of Theorem \ref{t:1} is obtained by reducing the problem (after some additional combinatorial arguments are applied)
to another problem in additive combinatorics, which may be interesting in its own right.

Let $[-r,r]^d$ denote the set of $d$-dimensional vectors over $\{-r,\ldots,r\}$.
The main tool used in the proof of Theorem \ref{t:1} is the following theorem.
\begin{theorem}
\label{t:2}
For positive integers $r,d,q$ the following holds.
Any sequence of $n \ge (\lceil q/r \rceil +2)(2rd+1)^d$ elements of $[-r,r]^d$ whose sum is in $[-q,q]^d$ contains a nonempty proper subsequence whose sum is zero.
\end{theorem}
It is noteworthy to mention here a result of Erd\H{o}s et al. \cite{ECRS-1993} who proved that for every positive integer $k$, there is a number $N(k)$
such that for all $n > N(k)$, if the edges of the complete graph $K_n$ are red-blue colored, then there is a monochromatic complete subgraph
$K_k$ (say blue) the degree of whose vertices in the blue-colored graph differ by at most $R(k)-2$ where $R(k)$ is the diagonal Ramsey number.
Our proof of Theorem \ref{t:1} shows, in fact, that for $n > N(K)$ and any red-blue coloring, we can delete some $M(k)$ vertices
such that there is a monochromatic $K_k$ (say blue) with equal degrees in the blue-colored graph.

We finally mention that our result here shows that any graph $G$ is ``close'' to another graph $G'$ with $k$ vertices of the same degree, where
by close we mean that few vertices need to be removed. It is of course natural to ask what happens if we are allowed to remove (or also add) edges.
As it turns out this variant of the problem is not hard. Indeed, it is easy to see that in this case it is enough to remove $O(k^2)$ edges in order to force $k$ vertices of the same degree. It is also easy to see that $\Omega(k^2)$ edge removals are required. We omit the details.

The rest of this paper is organized as follows. Section \ref{s:t2} contains the proof of Theorem \ref{t:2}.
The proof of Theorem \ref{t:1} is given in Section \ref{s:t1}.
Section \ref{s:c3} considers the special case $C(3)$.
A non-linear lower bound for $C(k)$ is established in Section \ref{s:lower}.

\section{Proof of Theorem \ref{t:2}}\label{s:t2}

Our proof of Theorem \ref{t:2} uses a geometric approach similar to the one used by Alon and Berman in \cite{AB-1986}.
To this end, we need the following result of Sevast'yanov \cite{sevastyanov-1978}.
\begin{lemma}[Sevast'yanov \cite{sevastyanov-1978}]\label{l:2.1}
Let $V$ be any normed $d$-dimensional space. Suppose $v_1,\ldots,v_n \in V$ where $\|v_i\| \le 1$ and $\sum_{i=1}^n v_i =0$.
Then there is a permutation $\pi$ on $\{1,\ldots,n\}$ such that for all $j=1,\ldots,n$,
$$
\left\|\sum_{i=1}^j v_{\pi(i)}\right\| \le d\;.
$$
\end{lemma}

\vspace{5pt}
\noindent
\begin{prevproof}{Theorem}{t:2}
Consider a sequence $X=[x_1,\ldots,x_n]$ of $n \ge(\lceil q/r \rceil +2)(2rd+1)^d$ elements of $[-r,r]^d$ whose sum is $w \in [-q,q]^d$.
It is easy to see that since $w \in [-r,r]^d$, we can always find $p=\lceil q/r \rceil$ vectors $x_{n+1},\ldots,x_{n+p} \in [-r,r]^d$ whose sum is $-w$. Adding these ``artificial'' vectors to $X$ we obtain a zero-sum collection $X'=[x_1,\ldots,x_n,x_{n+1},\ldots,x_{n+p}]$.
For every $1 \leq i \leq n+p$ set $v_i = x_i/r$, and consider $X''=[v_1,\ldots,v_{n+p}]$. Then $X''$ is a zero-sum sequence with $\|v_i\|_{\infty} \le 1 $ for every $v_i \in X''$.
By Lemma \ref{l:2.1} (with the $\ell_{\infty}$ norm), there is a permutation $\pi$ on $\{1,\ldots,n+p\}$ such that for all $j=1,\ldots,n+p$, we have
$\left\|\sum_{i=1}^j v_{\pi(i)}\right\|_{\infty} \le d$.
Observe now that each coordinate of any $v_i$ is a rational of the form $t/r$ where $t \in \{-r,\ldots,r\}$.
Hence, any possible sum of a subset of vectors of $X''$ whose $\ell_{\infty}$ norm is at most $d$ is a vector with rational coordinates
where each coordinate is of the form $t/r$ where $t \in \{-rd \ldots, rd\}$. Hence, there are at most $(2rd+1)^d$ such possible sums.
As any prefix sum of the elements of $X''$ ordered by the permutation $\pi$ has $\ell_{\infty}$ norm at most $d$, we have, by the pigeonhole principle,
that some prefix sum value repeats at least $(n+p)/(2rd+1)^d$ times. Since $n \ge (p+2)(2rd+1)^d$, some prefix sum value, call it $z$, repeats
at least $p+2$ times.

Consider therefore $p+2$ locations $j_1,\ldots,j_{p+2}$ for which $\sum_{i=1}^{j_\ell} v_{\pi(i)} = z$ for $\ell=1,\ldots,p+2$.
This means that for $\ell=1,\ldots,p+1$, we have
$$
\sum_{i=j_\ell+1}^{j_{\ell+1}} v_{\pi(i)} = 0\;.
$$
As we added only $p$ artificial vectors to $X'$ (and thus also to $X''$), one of these $p+1$ collections must give us the collection we want.
Specifically, there exists some $\ell$ for which  ${\pi(j_{\ell}+1)},\ldots,{\pi(j_{\ell+1})} \in \{1,\ldots,n\}$, implying that we have a non-empty and proper collection of vectors from $X$ (namely $x_{\pi(j_{\ell}+1)},\ldots,x_{\pi(j_{\ell+1})}$) whose sum is zero.
\end{prevproof}

The following is an immediate corollary of Theorem \ref{t:2}.
\begin{corollary}
\label{c:2}
For positive integers $r,d,q$ the following holds.
Any sequence of $n \ge (\lceil q/r \rceil +2)(2rd+1)^d$ elements of $[-r,r]^d$ whose sum, denoted by $z$, is in $[-q,q]^d$ contains a subsequence
of length at most $(\lceil q/r \rceil +2)(2rd+1)^d$ whose sum is $z$.
\end{corollary}

\section{Proof of Theorem \ref{t:1}}\label{s:t1}

We start with the following simple lemma.
\begin{lemma}
\label{l:pigeon}
Suppose that $G$ is a complete graph with $n$ vertices whose edges have weights in $\{0,\ldots,r\}$.
If $n \ge s^2$ where $s$ is a positive integer, then there is a set $S$ of $s$ vertices with
the property that $|deg(x)-deg(y)| \le sr$ for any $x,y \in S$.
\end{lemma}
\begin{proof}
Clearly, the difference between the minimum degree of $G$ and the maximum degree of $G$ is at most $(n-2)r$.
We need to prove that for any positive integer $s$, if $n \ge s^2$, then there is a set $S$ of $s$ vertices with
the property that $|deg(x)-deg(y)| \le sr$ for any $x,y \in s$.
To see this, assume that the vertices are $\{v_1,\ldots,v_n\}$ where $deg(v_i) \le deg(v_{i+1})$.
Let $S_j = \{v_{(s-1)j+1},\ldots,v_{(s-1)(j+1)+1}\}$ for $j=0,\ldots,\lfloor (n-1)/(s-1) \rfloor -1$.
Observe that the last vertex of $S_j$ is also the first vertex of $S_{j+1}$.
Now, if some $S_j$ has the desired property of $S$, we are done.
Otherwise the difference between $deg(v_1)$ and the last vertex of $S_{\lfloor (n-1)/(s-1) \rfloor-1}$ is at least
$$
(sr+1) \cdot \lfloor \frac{n-1}{s-1} \rfloor \ge (sr+1)\frac{n-s+1}{s-1} > (n-2)r
$$
where the last inequality follows from the assumption $n \ge s^2$. As this contradicts the maximum gap between the minimum
and maximum degree, the lemma follows.
\end{proof}
We note that for the case $r=1$, a slightly stronger result follows from a result of Erd\H{o}s et al \cite{ECRS-1993}
using the Erd\H{o}s-Gallai criterion for degree sequences. It is proved that for any $n \ge s \ge 2$ there is a set $S$ of $s$ vertices with the property that
$|deg(x)-deg(y)| \le s-2$ for any $x,y \in S$.

Recall that $R_r(k)$ is the {\em multicolored Ramsey number}, which is the least integer $s$ such that in any coloring of the edges of the
complete graph with $s$ vertices using $r$ colors, there is a monochromatic $K_k$ (see \cite{radziszowski-2006}).

\begin{prevproof}{Theorem}{t:1}
Let $s=R_{r+1}(k)$ be the multicolored Ramsey number, and set
$$
C = C(k,r) =  \max\{s^2, k+(s+2)(2r(k-1)+1)^{k-1}\}\;.
$$
Consider a complete graph with $n$ vertices whose edges have weights in $\{0,\ldots,r\}$.
We may assume that $n \ge C$ as otherwise we can just delete, say, $n-2 \le C$ vertices and obtain a regular subgraph with two vertices.
Since $n \ge C \ge s^2$, we have by Lemma \ref{l:pigeon} that there is a set $S$ of $s$ vertices
with the property that $|deg(x)-deg(y)| \le sr$ for any $x,y \in S$. Thus, for some $p \in \{0,\ldots,r\}$ the induced complete subgraph $G[S]$ has
an induced complete subgraph on $k$ vertices, denoted by $K$, such that the edge $(x,y)$ for $x,y \in K$ has weight $p$.

Without loss of generality, assume that $V(G) \setminus K = \{v_1,\ldots,v_{n-k}\}$ and $K=\{v_{n-k+1},\ldots,v_n\}$.
Let $w(v_i,v_j)$ denote the weight of the edge connecting $v_i$ and $v_j$.
We construct a sequence of $n-k$ vectors $x_1,\ldots,x_{n-k}$ of $[-r,r]^{k-1}$ as follows.
Coordinate $j$ of $x_i$ is $w(v_{n-k+j},v_i)-w(v_n,v_i)$ for $i=1,\ldots,n-k$ and $j=1,\ldots,k-1$.
Observe that indeed $w(v_{n-k+j},v_i)-w(v_n,v_i) \in \{-r,\ldots,r\}$ as required.
What can be said about the sum of all the $j$'th coordinates?
\begin{eqnarray*}
\sum_{i=1}^{n-k} (w(v_{n-k+j},v_i)-w(v_n,v_i)) & = &\sum_{i=1}^{n-k}w(v_{n-k+j},v_i) - \sum_{i=1}^{n-k}w(v_n,v_i) \\
& = & (deg(v_{n-k+j})-p(k-1)) - (deg(v_n)-p(k-1)) \\
& = & deg(v_{n-k+j}) - deg(v_n) \le sr\;.
\end{eqnarray*}
Hence,
$$
z = \sum_{i=1}^{n-k} x_i \in [-sr,sr]^{k-1}\;.
$$
Since $n-k \ge C - k \ge (s+2)(2r(k-1)+1)^{k-1}$, we have by Corollary \ref{c:2} with $d=k-1$ and $q=sr$ that there is a subsequence of $X$ of size at most
$(s+2)(2r(k-1)+1)^{k-1}$ whose sum is $z$.
Deleting the vertices of $G$ corresponding to the elements of this subsequence results in a subgraph with at least $n-(s+2)(2r(k-1)+1)^{k-1} \ge n-C$ vertices
in which all the $k$ vertices of $K$ have the same degree.
\end{prevproof}

\noindent
The following corollary is immediate from the bound $C(k,1) \le \max\{s^2, k+(s+2)(2(k-1)+1)^{k-1}\}$ given in the proof of Theorem \ref{t:1}, together with the well known facts that $s=R_2(k) < 4^k$ and $R_2(3)=6$.
\begin{corollary}\label{c:explicit}
We have $C(k) \le (8k)^k$ and $C(3) \le 203$.
\end{corollary}

\section{Equating three vertices}\label{s:c3}

In this section we prove that $3 \le C(3)  \le 6$.
\begin{prop}
For any graph $G$ with at least $5$ vertices, one can delete at most $6$ vertices such that the subgraph obtained has at least three vertices with the same degree.
Consequently, $C(3) \le 6$.
\end{prop}
\begin{proof}
By a result of \cite{ECRS-1993}, $G$ has a set $X$ of $5$ vertices such that $|deg(u)-deg(v)| \le 3$ for $u,v \in X$.
Now, either $X$ has a triangle, or else $X$ has an independent set of size $3$, or else $X$ induces a $C_5$.
In any case, this implies that $X$ has three vertices $x,y,z$ with $deg(x) \le deg(y) \le deg(z)$ such that the following holds:
$(x,y)$ is an edge if and only if $(x,z)$ is an edge. Furthermore, if $(y,z)$ is an edge, then $\{x,y,z\}$ induce a triangle.
We use this property implicitly throughout the remainder of the proof.
We may also assume $deg(z) > deg(x)$ otherwise $x,y,z$ already have the same degree.
Throughout the proof we denote by $N(.)$ the set of neighbors of a vertex in the current $G$ (that is, in the graph $G$ after some vertices have possibly
been deleted). Similarly, we denote by $deg(.)$ the degree of a vertex in the current $G$.

Consider first the case $deg(x) < deg(y) = deg(z)$ in the original $G$.
If $(N(z) \setminus N(x)) \cap (N(y) \setminus N(x)) \neq \emptyset$ we can delete a vertex of this intersection
and decrease the degrees of $y$ and $z$ by $1$ without affecting the degree of $x$.
Otherwise, if $(N(z) \setminus N(x)) \cap (N(y) \setminus N(x)) = \emptyset$ we can delete a vertex of
$N(z) \setminus N(x)$ and a vertex of $N(y) \setminus N(x)$ and decrease the degrees of $y$ and $z$ by $1$ without affecting the degree of $x$.
Observe that in any case we delete at most two vertices. Repeating this process at most three times we eventually obtain
$deg(x)=deg(y)=deg(z)$. Overall, we have deleted at most $6$ vertices.

Consider next the case $deg(x) \le deg(y) < deg(z)$ in the original $G$.
Let $p=deg(z)-deg(y)$ and let $q=deg(y)-deg(x)$, and observe that $p+q \le 3$.
Let us first equate $deg(z)$ and $deg(y)$ by deleting
some $u \in N(z) \setminus N(y)$. Observe that $u \neq x$. We always prefer to delete a vertex $u$ that is non-adjacent to $x$, as long as there is
such a vertex $u$. Overall, we have deleted $p$ vertices. The problem is that in the current graph we may have
that $deg(x)$ also decreased by some amount $r \le p$. Suppose first that $r=0$.
As in the previous case, we may need to delete $2q$ additional vertices to equate the degrees of $y$ and $z$ to that of $x$.
Overall, we have deleted at most $p+2q \le 6$ vertices.
If $r > 0$ then this means that at some point, when we deleted a vertex $u$, that vertex also
had to be adjacent to $x$. Hence, in the current graph, $(N(z) \setminus N(x)) \subset (N(y) \setminus N(x))$.
So, we may simply delete $r$ additional vertices (all of them from $N(z) \setminus N(x)$) to equate the degrees of $y$ and $z$ to that of $x$.
Overall, we deleted $p+r \le 2p \le 6$ vertices.
\end{proof}

\begin{figure}[ht]
\begin{center}
\includegraphics[width = 3.8in]{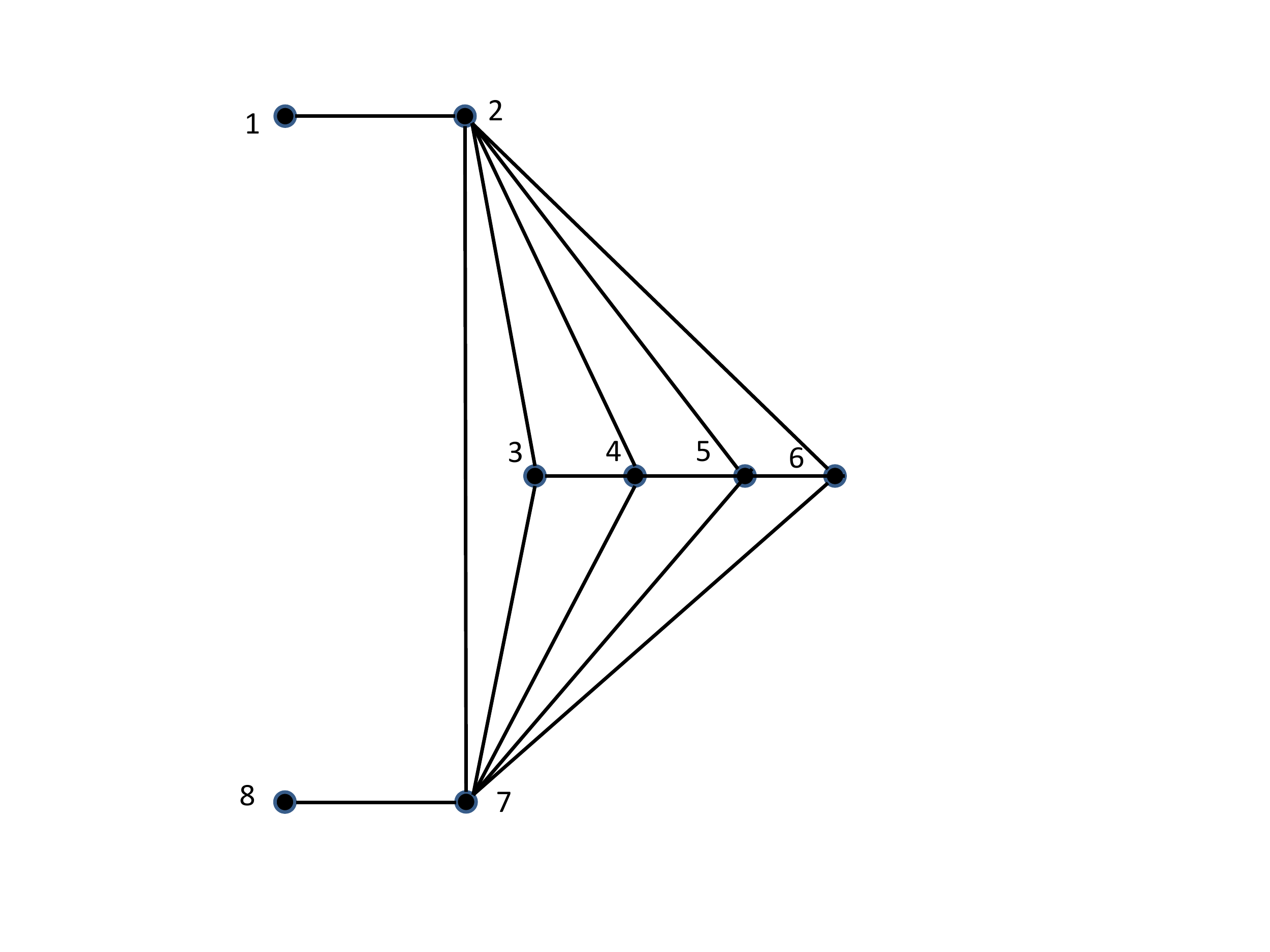}
\caption{\small A graph showing that $C(3) \ge 3.$
\label{f:graph}}
\end{center}
\end{figure}

The graph in Figure \ref{f:graph} proves that $C(3) \ge 3$. If we delete vertices $1,2,8$, then in the resulting graph,
vertices $4,5,7$ have the same degree. It is easy to check that deleting any two vertices does not yield a graph
with repetition number $3$. In fact, a computer verification asserts that this graph is the smallest graph (in terms of the number
of vertices) for which we need to delete more than two vertices to obtain repetition number $3$.

\section{A lower bound for $C(k)$}\label{s:lower}

In this section we prove that $C(k) = \Omega(k \log k)$.
Our construction is based upon an additional building block: graphs that have the property that
all of their induced subgraphs have repetition number which is not too large.
\begin{lemma}
\label{l:goodgraph}
For any integer $n$, there exists a graph $D_n$ with $n$ vertices such that no induced subgraph of $D_n$ has more than
$3n/\ln n$ vertices with the same degree.
\end{lemma}
\begin{proof}
Consider the non-increasing sequence of integers $a_i=\lceil 2n/(i\ln n) \rceil$.
Let $S_i= \sum_{j=1}^i a_j$ and let $s$ be the largest integer such that $S_s \le n$.
Observe that $s \le \sqrt{n}$ since
$$
S_{\lceil \sqrt{n} \rceil} = \sum_{i=1}^{\lceil \sqrt{n} \rceil} a_i \ge \frac{2n}{\ln n} \left(\sum_{i=1}^{\lceil \sqrt{n} \rceil}\frac{1}{i}\right) > n\;.
$$
Note that we also have $n-S_s \le a_1 = \lceil 2n/(\ln n) \rceil$.

Construct a graph $D_n$ with $n$ vertices as follows. Take $s$ vertex-disjoint cliques $B_1,\ldots,B_s$ where $B_i=K_{a_i}$ is a clique on
$a_i$ vertices. The remaining set of $n-S_s$ vertices of $D_n$ is a set denoted by $B_{s+1}$ consisting of $n-S_s$ isolated vertices.

Let $H$ be any induced subgraph of $D_n$, and suppose that $K$ is a set of $k$ vertices
of $H$ with the same degree in $H$. Let $j$ be the largest index such that $K$ contains a vertex from $B_j$. Observe that if $j=s+1$, then
all vertices of $K$ have degree $0$ in $H$ and thus contain at most one vertex from each $B_i$ for $i \ge 1$.
In this case we have
$$
k \le |B_0| + s = n-S_s+s \le \lceil 2n/(\ln n) \rceil + \sqrt{n} \le \frac{3n}{\ln n}\;.
$$
If $j \le s$, then all vertices of $K$ have the same degree in $H$ which is at most $a_j-1$ and thus $K$ contains at most $a_j$ vertices from each $B_i$
for $i=1,\ldots,j$. Hence
$$
k \le j a_j = j\lceil 2n/(j\ln n) \rceil \le \frac{2n}{\ln n}+\sqrt{n} \le \frac{3n}{\ln n}\;.
$$
\end{proof}

\begin{theorem}
\label{t:3}
$C(k) = \Omega(k \log k)$.
\end{theorem}
\begin{proof}
We will prove the slightly stronger assertion that for $k$ sufficiently large, there are infinitely many graphs such that one cannot equate
the degrees of $k$ vertices by removing fewer than $(k \ln k)/10$ vertices.

Let $H_n$ be any graph with $n$ vertices and with $rep(H_n)=2$. Denote is vertices by $\{h_1,\ldots,h_n\}$ where $deg(h_i) \le deg(h_{i+1})$.
Let $H_n(q)$ be the graph obtained from $H_n$ by the following procedure.
First, take a $q$-blowup of $H_n$, namely, each $h_i$ is replaced by an independent set of $q$ vertices, denoted by $B_i=\{h_{i,1},\ldots,h_{i,q}\}$.
We connect each vertex of $B_i$ to each vertex of $B_j$ if and only if $(h_i,h_j) \in E(H_n)$. Otherwise, $B_i \cup B_j$ is an independent set.
Next, for each $i=1,\ldots,n$ we place the graph $D_q$ of Lemma \ref{l:goodgraph} inside $B_i$.
Observe that $H_n(q)$ has $nq$ vertices, and, furthermore, since $rep(H_n) = 2$, we have that for $i > j$, $deg(h_i)-deg(h_j) \ge \lfloor (i-j)/2 \rfloor$.
Since $deg(h_i) q \le deg(h_{i,x}) < (deg(h_i)+1)q$ for all $x=1,\ldots,q$, we have that for $i > j$,
\begin{equation}
\label{e:q}
deg(h_{i,x})-deg(h_{j,y}) \ge \frac{q(i-j)}{2}-\frac{3q}{2}\;.
\end{equation}

We claim that for, say, $q=5k$, any induced graph obtained from $H_n(q)$ by deleting fewer than $(k\ln k)/10$ vertices, does not have $k$ vertices with the same degree.
Indeed, assume otherwise and suppose we can delete a set of vertices $X$ of size at most $(k \ln k)/10$ from $H_n(q)$ such that the resulting subgraph $H'$ has a set $K$ of $k$ vertices with the same degree.
Let $B'_i=B_i \setminus X$ for $i=1,\ldots,n$ and notice that if two vertices of $K$ belong to $B_i$, then they have the same degree in the subgraph induced by $B'_i$.
Notice also that in $H_n(q)$, the degrees of any two vertices of $K$ differ by at most $(k \ln k)/10$. Hence, if $v \in B_i$ and $u \in B_j$ are two vertices of $K$,
then $|i-j| \le (\ln k)/20-1$, as otherwise, by (\ref{e:q}), their degrees in $H_n(q)$ would differ by at least
$$
\frac{q(\ln k/20-1)}{2}-\frac{3q}{2} > \frac{k \ln k}{10}
$$
where in the last inequality we used $q=5k$ and that $k$ is sufficiently large.
It follows that some $B'_i$ contains at least $20k/\ln k$ vertices of $K$. However, as the subgraph induced by $B'_i$ was obtained from $D_q$ by deleting vertices,
we have by Lemma \ref{l:goodgraph} that it cannot have more than $3q/ \ln q$ vertices of the same degree.
But we now arrive at a contradiction since
$3q/ \ln q < 20k/\ln k $.
\end{proof}

\bibliographystyle{plain}

\bibliography{samedeg}

\end{document}